\documentclass{amsart}
\usepackage{eurosym}
\usepackage{amsfonts}

\newtheorem{theorem}{Theorem}
\theoremstyle{plain}

\newtheorem{definition}{Definition}

\newtheorem{lemma}{Lemma}

\newtheorem{remark}{Remark}

\numberwithin{equation}{section}

\begin{document}
\title[]{General decay and blow-up of solutions for a nonlinear wave
equation with memory and fractional boundary damping terms}
\author{Salah Boulaaras$^{\text{1,2,*}}$, Fares Kamache$^{\text{4}}$, Youcef
Bouizem$^{\text{3}}$ and Rafik Guefaifia$^{\text{4}}$}
\address{$^{\text{1}}$Department of Mathematics, College of Sciences and
Arts, Al-Rass, Qassim University, Kingdom of Saudi Arabia.\\
$^{\text{2}}$Laboratory of Fundamental and Applied Mathematics of Oran
(LMFAO), University of Oran 1, Ahmed Benbella. Algeria.\\
$^{\text{3}}$Department of Mathematics, Faculty of Mathematics and
Informatics, USTOMB, Oran, Algeria\\
$^{\text{4}}$Department of Mathematics, Faculty of Exact Sciences,
University Tebessa, Tebessa, Algeria}
\email{Saleh\_boulaares@yahoo.fr}
\dedicatory{}
\thanks{*Corresponding author: saleh\_boulaares@yahoofr}

\begin{abstract}
The paper studies the global existence and general decay of solutions using
Lyaponov functional for a nonlinear wave equation, taking into account the
fractional derivative boundary condition and memory term. In addition, we
establish the blow up of solutions with nonpositive initial energy.
\end{abstract}

\maketitle

\section{Introduction}

Extraordinary differential equations, also known as fractional differential
equations are a generalization of differential equations through fractional
calculus. Much attention has been accorded to fractional partial
differential equations during the past two decades due to the many chemical
engineering, biological, ecological and electromagnetism phenomena that are
modeled by initial boundary value problems with fractional boundary
conditions. See Tarasov \cite{18}, Magin \cite{13}, and Val\'{e}rio et al
\cite{19}. \newline
In this work we consider the nonlinear wave equation
\begin{equation}
\left\{
\begin{array}{ll}
u_{tt}-\Delta u+au_{t}+\int_{0}^{t}g\left( t-s\right) \Delta u\left(
s\right) ds=|u|^{p-2}u, & x\in \Omega ,t>0, \\
\frac{\partial u}{\partial \nu }=-b\partial _{t}^{\alpha ,\eta }u, & x\in
\Gamma _{0},t>0, \\
u=0, & x\in \Gamma _{1},t>0, \\
u(x,0)=u_{0}(x),\quad u_{t}(x,0)=u_{1}(x), & x\in \Omega ,%
\end{array}%
\right.  \label{1.1}
\end{equation}%
where $\Omega $ is a bounded domain in $%
\mathbb{R}
^{n}$, $n\geq 1$ with a smooth boundary $\partial \Omega $ of class $C^{2}$
and $\nu $ is the unit outward normal to $\partial \Omega =\Gamma _{0}\cup
\Gamma _{1}$, where $\Gamma _{0}$ and $\Gamma _{1}$ are closed subsets of $%
\partial \Omega $ with $\Gamma _{0}\cap \Gamma _{1}=\emptyset $.\newline
$a,b>0$, $\ p>2$, and $\partial _{t}^{\alpha ,\eta }$ with $0<\alpha <1$ is
the Caputo's generalized fractional derivative (see \cite{7} and \cite{8})
defined by:
\begin{equation*}
\partial _{t}^{\alpha ,\eta }u(t)=\frac{1}{\Gamma (1-\alpha )}%
\int_{0}^{t}(t-s)^{-\alpha }e^{-\eta (t-s)}u_{s}(s)ds,\quad \eta \geq 0,
\end{equation*}%
where $\Gamma $ is the usual Euler gamma function. It can also be expressed
by
\begin{equation}
\partial _{t}^{\alpha ,\eta }u(t)=I^{1-\alpha ,\eta }u^{\prime }(t),
\label{1.2}
\end{equation}%
where $I^{\alpha ,\eta }$ is the exponential fractional integro-differential
operator given by%
\begin{equation*}
I^{\alpha ,\eta }u(t)=\frac{1}{\Gamma (\alpha )}\int_{0}^{t}(t-s)^{\alpha
-1}e^{-\eta (t-s)}u(s)ds,\quad \eta \geq 0.
\end{equation*}

In the context of boundary dissipations of fractional order problems, the
main research focus is on asymptotic stability of solutions starting by
writing the equations as an augmented system (see \cite{15}). Then, various
techniques are used such as LaSalle's invariance principle and multiplier
method mixed with frequency domain, (see \cite{2}, \cite{3}, \cite{6}, \cite%
{7}, \cite{8}, \cite{16}, \cite{18}).\newline
In \cite{2}, Akil and Wehbe used semigoup theory of linear operators to
prove stability of the following problem
\begin{equation*}
\left\{
\begin{array}{ll}
u_{tt}-\Delta u=0, & x\in \Omega ,\quad t>0, \\
\frac{\partial u}{\partial \nu }=-b\partial _{t}^{\alpha ,\eta }u, & x\in
\Gamma _{0},\quad t>0,\quad \eta \geq 0,\quad 0<\alpha <1, \\
u=0, & x\in \Gamma _{1},\ t>0, \\
u(x,0)=u_{0}(x),\ u_{t}(x,0)=u_{1}(x), & x\in \Omega .%
\end{array}%
\right.
\end{equation*}

In \cite{14}, Mbodje carried on the study by investigating the decay rate of
energy to prove strong asymptotic stability if $\eta =0$, and a polynomial
decay rate $E(t)\leq \displaystyle\frac{c}{t}$ if $\eta >0$. \newline

Later in \cite{11}, Kirane and Tatar proved global existence and exponential
decay of the following wave equation with mild internal dissipation

\begin{equation}
\left\{
\begin{array}{ll}
u_{tt}(x,t)-\Delta u(x,t)+au_{t}(x,t)+\int_{0}^{t}g\left( t-s\right) \Delta
u\left( s\right) ds=f(x,t), & x\in \Omega ,t>0, \\
\frac{\partial u}{\partial \nu }(x,t)+%
\int_{0}^{t}K(x,t-s)u_{s}(x,s)ds=h(x,t), & x\in \Gamma _{0},t>0, \\
u_{0}(x,t)=0 & x\in \Gamma _{1},t>0, \\
u(x,0)=u_{0}(x)\quad u_{t}(x,0)=u_{1}(x) & x\in \Omega .%
\end{array}%
\right.  \label{1.3}
\end{equation}%
where the homogeneous case was also considered in \cite{4} by Alabau and al,
in order to establish polynomial stability, then in \cite{5} for exponential
decay.\newline
Dai and Zhang \cite{8} replaced $\int_{0}^{t}K(x,t-s)u_{s}(x,s)ds$ by $%
\partial _{t}^{\alpha}u(x,t)$ and $h(x,t)$ by $|u|^{m-1}u(x,t)$, and managed
to prove exponential growth for the same problem.

Noting that the nonlinear wave equation with boundary fractional damping
case was first considered by authors in \cite{R20}, where they used the
augmented system to prove the exponential stability and blow up of solutions
in finite time. \newline

Motivated by our recent work in \cite{R20} and based on the construction of
a Lyapunov function, we prove in this paper under suitable conditions on the
initial data the stability of a wave equation with fractional damping and
memory term. This technique of proof was recently used by \cite{9} and \cite%
{R20} to study the exponential decay of a system of nonlocal singular
viscoelastic equations. \newline
Here we also consider three different cases on the sign of the initial
energy as recently examined by Zarai and al \cite{20}, where they studied
the blow up of a system of nonlocal singular viscoelastic equations.\newline

The organization of our paper is as follows. We start in sect.2 by giving
some lemmas and notations in order to reformulate our problem (1.1) into an
augmented system. In the following section, we use the potential well theory
to prove the global existence result. Then, the general decay result in
section 4. In sect.5, following a direct approach, we prove blow up of
solutions.

\section{Preliminaries}

Let us introduce some notations, assumptions, and lemmas that are effective
for proving our results. \newline

Assume that the relaxation function $g$ satisfies

$\left( G_{1}\right) $ $g:%
\mathbb{R}
_{+}\rightarrow
\mathbb{R}
_{+}$ is a nonincreasing differentiable function with
\begin{equation}
g\left( 0\right) >0,\text{ \ \ \ \ \ }1-\int_{0}^{\infty }g\left( s\right)
ds=l>0  \label{2.1}
\end{equation}

$\left( G_{2}\right) $ There exists a constant $\xi >0$ such that%
\begin{equation}
g^{\prime }\left( t\right) \leq -\xi g\left( t\right) ,\text{ \ \ \ \ \ }%
\forall t>0.  \label{2.2}
\end{equation}

We denote

\begin{equation}
\left( g\circ u\right) \left( t\right) =\int_{0}^{t}g\left( t-s\right) %
\Arrowvert u\left( t\right) -u\left( s\right) \Arrowvert^{2}ds,  \label{2.3}
\end{equation}
and
\begin{equation*}
\aleph =\{w\in H_0^1\vert I(w)>0\} \cup\{0\},
\end{equation*}

\begin{equation*}
H_{\Gamma_1} ^1 (\Omega)=\left\{u\in H ^1 (\Omega),u \vert _{\Gamma_1}
=0\right\}.
\end{equation*}

\begin{lemma}
(Sobolev-Poincar\'{e} Inequality, see \cite{16})\newline
If either $1\leq q\leq \frac{N+2}{N-2}$, $\left( N\geq 3\right) $ or $1\leq
q\leq +\infty $ $\left( N=2\right) $. Then there exists $C_{\ast }>0$ such
that
\begin{equation*}
\Vert u\Vert _{q+1}\leq C_{\ast }\Vert \nabla u\Vert _{2}, \ \ \forall u\in
H_{0}^{1}(\Omega ),
\end{equation*}
\end{lemma}

\begin{lemma}
(Trace -Sobolev embedding ) \newline
For all $p$ such that
\begin{equation}
2<p\leq \frac{2(n-1)}{n-2}  \label{2.4}
\end{equation}%
we have
\begin{equation*}
H_{\Gamma _{1}}^{1}(\Omega )\hookrightarrow L^{p}(\Gamma _{0}).
\end{equation*}%
We denote by $B_{q}$ the embedding constant i.e.,
\begin{equation*}
\Arrowvert u\Arrowvert_{p,\Gamma _{0}}\leq B_{q}\Arrowvert u\Arrowvert_{2}.
\end{equation*}
\end{lemma}

\begin{lemma}
( \cite{20}, p. 5, Lemma 2 or \cite{12}, p. 1406 ,Lemma 4.1) \newline
Consider a nonegative function $B(t)\in C^{2}(0,\infty )$ satisfying
\begin{equation}
B^{\prime \prime }(t)-4(\delta +1)B^{\prime }(t)+4(\delta +1)B(t)\geq 0,
\label{2.5}
\end{equation}%
where $\delta >0$.\newline
If
\begin{equation}
B^{\prime }(0)>r_{2}B(0)+l_{0},  \label{2.6}
\end{equation}%
then
\begin{equation}
B^{\prime }(t)\geq l_{0}, \ \ \forall t >0  \label{2.7}
\end{equation}%
where $l_{0} \in \mathbb{R}$, $r_2$ represents the smallest root of the
equation
\begin{equation}
r^{2}-4(\delta +1)r+(\delta +1)=0.  \label{2.8}
\end{equation}
i.e. $r_{2}=2(\delta +1)-2\sqrt{(\delta +1)\delta }.$
\end{lemma}

\begin{lemma}
(\cite{20}, p. 5, Lemma 3 or \cite{12}, p. 1406 ,Lemma 4.2)\newline
Let $J\left(t\right) $ be a non-increasing function on $\ \left[
t_{0},\infty \right) $ verifying the differential inequality
\begin{equation}
J^{\prime }\left( t\right) ^{2}\geq \alpha +bJ\left( t\right) ^{2+\frac{1}{%
\delta }},\ \ t\geq t_{0} \geq 0,  \label{2.9}
\end{equation}%
\newline
where $\alpha >0,\ b\in
\mathbb{R}
,$\ then there exists $T^{\ast } >0$\ such that
\begin{equation}
\lim_{t\rightarrow T^{\ast -}}J\left( t\right) =0,  \label{2.10}
\end{equation}%
\newline
with the following upper bound cases for $T^{\ast }$ \newline

$\mathbf{(i)}$ When $b<0$ and $J(t_{0})<\min \left\{ 1,\sqrt{\alpha /(-b)}%
\right\} $%
\begin{equation}
T^{\ast }\leq t_{0}+\frac{1}{\sqrt{-b}}\ln \frac{\sqrt{\frac{\alpha }{-b}}}{%
\sqrt{\frac{\alpha }{-b}}-J(t_{0})}.  \label{2.11}
\end{equation}%
\newline
$\mathbf{(ii)}$ When $b=0,$
\begin{equation}
T^{\ast }\leq t_{0}+\frac{J(t_{0})}{\sqrt{\alpha }}.  \label{2.12}
\end{equation}%
\newline
$\mathbf{(iii)}$ When $b>0,$
\begin{equation}
T^{\ast }\leq \frac{J(t_{0})}{\sqrt{\alpha }}  \label{2.13}
\end{equation}%
\newline
or
\begin{equation}
T^{\ast }\leq t_{0}+2^{\frac{3\delta +1}{2\delta }}\frac{\delta c}{\sqrt{%
\alpha }}\left( 1-\left[ 1+cJ(t_{0})\right] ^{\frac{1}{2\delta }}\right) ,
\label{2.14}
\end{equation}%
\newline
where%
\begin{equation*}
c=\left( \frac{b}{\alpha }\right) ^{\delta /\left( 2+\delta \right) }.
\end{equation*}
\end{lemma}

\begin{definition}
We say that $u$ is a blow-up solution of $\ (1.1)$ at finite time $T^{\ast }$
if
\begin{equation}
\lim_{t\rightarrow T^{\ast -}} \dfrac{1}{\left( \Arrowvert\nabla u\Arrowvert%
_{2}\right)}=0.  \label{2.15}
\end{equation}
\end{definition}

\begin{theorem}
(\cite{14}, Theorem 1) \newline
Consider the constant
\begin{equation*}
\varrho =(\pi )^{-1}\sin {(\alpha \pi )}
\end{equation*}%
and the function $\mu $ given by
\end{theorem}

\begin{equation}
\mu (\xi )=|\xi |^{\frac{(2\alpha -1)}{2}},\quad 0<\alpha <1, \ \xi \in
\mathbb{R}
.  \label{2.16}
\end{equation}%
Then, we can obtain
\begin{equation}
O=I^{1-\alpha ,\eta }U.  \label{2.20}
\end{equation}%
which is a relation between $U$ the \textquotedblright
input\textquotedblright\ of the system
\begin{equation}
\partial _{t}\phi (\xi ,t)+(\xi ^{2}+\eta )\phi (\xi ,t)-U(L,t)\mu (\xi )=0,
\ \ t>0,\eta \geq 0, \text{ }\xi \in
\mathbb{R}
\label{2.17}
\end{equation}
and the \textquotedblright output\textquotedblright\ $O$ given by
\begin{equation}
O(t)=\varrho \int_{-\infty }^{+\infty }\phi (\xi ,t)\mu (\xi )d\xi ,\text{ }%
\xi \in
\mathbb{R}
,\ t>0.  \label{2.19}
\end{equation}

Now using $(1.2)$ and Theorem 1, the augmented system related to our system $%
(1.1)$ may be given by
\begin{equation}
\left\{
\begin{array}{ll}
u_{tt}-\Delta u+au_{t}+\int_{0}^{t}g\left( t-s\right) \Delta u\left(
s\right) ds=|u|^{p-2}u, & x\in \Omega ,t>0, \\
\partial _{t}\phi (\xi ,t)+(\xi ^{2}+\eta )\phi (\xi ,t)-u_{t}(x,t)\mu (\xi
)=0, & x\in \Gamma _{0},\xi \in
\mathbb{R}
,t>0, \\
\frac{\partial u}{\partial \nu }=-b_{1}\int_{-\infty }^{+\infty }\phi (\xi
,t)\mu (\xi )d\xi , & x\in \Gamma _{0},\xi \in
\mathbb{R}
,t>0, \\
u=0, & x\in \Gamma _{1},t>0, \\
u(x,0)=u_{0}(x),\quad u_{t}(x,0)=u_{1}(x), & x\in \Omega , \\
\phi (\xi ,0)=0, & \xi \in
\mathbb{R}
,%
\end{array}%
\right.  \label{2.21}
\end{equation}%
where $b_{1}= b \varrho $.

\begin{lemma}
(\cite{3}, p. 3, Lemma 2.1)\newline
For all $\lambda \in D_{\eta }= \left\{ \lambda \in \mathbf{%
\mathbb{C}
}:\Im m\lambda \neq 0\right\} \cup \left\{ \lambda \in \mathbf{%
\mathbb{C}
}:\Re e\lambda +\eta >0\right\}$, we have
\begin{equation*}
A_{\lambda }=\int_{-\infty }^{+\infty }\frac{\mu ^{2}(\xi )}{\eta + \lambda
+\xi ^{2}}d\xi =\frac{\pi }{\sin {(\alpha \pi )}}(\eta +\lambda)^{\alpha -1}.
\end{equation*}
\end{lemma}

\begin{theorem}
(Local existence and Uniqueness) \newline
Assume (2.4) holds. Then for all $(u_{0},u_{1},\phi _{0})\in H_{\Gamma
_{0}}^{1}(\Omega )\times L^{2}(\Omega )\times L^{2}(-\infty ,+\infty )$,
there exists some $T$ small enough such that problem ($\ref{2.21}$) admits a
unique solution
\begin{equation}
\left\{
\begin{array}{ll}
u\in C([0,T),H_{\Gamma _{0}}^{1}(\Omega )), &  \\
u_{t}\in C([0,T),L^{2}(\Omega )), &  \\
\phi \in C([0,T),L^{2}(-\infty ,+\infty ). &
\end{array}%
\right.  \label{2.22}
\end{equation}
\end{theorem}

\section{Global existence}

Before proving the global existence for problem (\ref{2.21}), let us
introduce the functionals:
\begin{equation*}
I(t)=\left( 1-\int_{0}^{t}g(s)ds\right) \Arrowvert\nabla u\Arrowvert%
_{2}^{2}+\left( g\circ \nabla u\right) \left( t\right) -\Arrowvert u%
\Arrowvert_{p}^{p}
\end{equation*}%
and
\begin{equation*}
J(t)=\frac{1}{2}\left[ \left( 1-\int_{0}^{t}g(s)ds\right) \Arrowvert\nabla u%
\Arrowvert_{2}^{2}+\left( g\circ \nabla u\right) \left( t\right) \right] -%
\frac{1}{p}\Arrowvert u\Arrowvert_{p}^{p}.
\end{equation*}%
The energy functional $E$ associated to system (\ref{2.21}) is given as
follows:
\begin{equation}
E(t)=\frac{1}{2}\Arrowvert u_{t}\Arrowvert_{2}^{2}+\frac{1}{2}\left(
1-\int_{0}^{t}g(s)ds\right) \Arrowvert\nabla u\Arrowvert_{2}^{2}+\frac{1}{2}%
\left( g\circ \nabla u\right) \left( t\right) -\frac{1}{p}\Arrowvert u%
\Arrowvert_{p}^{p}+\frac{b_{1}}{2}\int_{\Gamma _{0}}\int_{-\infty }^{+\infty
}|\phi (\xi ,t)|^{2}d\xi d\rho .  \label{3.1}
\end{equation}

\begin{lemma}
If $(u,\phi )$ is a regular solution to (\ref{2.21}), then the energy
functional given in $(3.1)$ verifies
\begin{equation}
\frac{d}{dt}E(t)=-a\Arrowvert u_{t}\Arrowvert_{2}^{2}-\frac{1}{2}g\left(
t\right) \Arrowvert\nabla u\Arrowvert_{2}^{2}+\frac{1}{2}\left( g^{\prime
}\circ \nabla u\right) \left( t\right) -b_{1}\int_{\Gamma _{0}}\int_{-\infty
}^{+\infty }(\xi ^{2}+\eta )|\phi (\xi ,t)|^{2}d\xi d\rho \leq 0.
\label{3.2}
\end{equation}
\end{lemma}

\begin{proof}
Multiplying by $u_{t}$ in the first equation from (\ref{2.21}), using
integration by parts over $\Omega $, we get
\begin{eqnarray*}
&&\frac{1}{2}\Arrowvert u_{t}\Arrowvert_{2}^{2}-\int_{\Omega }\Delta
uu_{t}dx+a\Arrowvert u_{t}\Arrowvert_{2}^{2}+\frac{1}{2}\left(
1-\int_{0}^{t}g(s)ds\right) \Arrowvert\nabla u\Arrowvert_{2}^{2}+\frac{1}{2}%
\left( g\circ \nabla u\right) \left( t\right) \\
&=& \int_{\Omega }|u|^{p-2}u\text{ }u_{t}dx.
\end{eqnarray*}%
Therefore
\begin{equation}
\begin{split}
\frac{d}{dt}\left[ \frac{1}{2}\Arrowvert u_{t}\Arrowvert_{2}^{2}\right. &
\left. +\frac{1}{2}\left( 1-\int_{0}^{t}g(s)ds\right) \Arrowvert\nabla u%
\Arrowvert_{2}^{2}+\frac{1}{2}\left( g\circ \nabla u\right) \left( t\right) -%
\frac{1}{p}\Arrowvert u\Arrowvert_{p}^{p}\right] \\
+a\Arrowvert u_{t}\Arrowvert_{2}^{2}& +b_{1}\int_{\Gamma
_{0}}u_{t}(x,t)\int_{-\infty }^{+\infty }\mu (\xi )\phi (\xi ,t)d\xi d\rho
=0.
\end{split}
\label{3.3}
\end{equation}%
Multiplying by $b_{1}\phi $ in the second equation from (\ref{2.21}), and
integrating over $\Gamma _{0}\times (-\infty ,+\infty )$, we get
\begin{equation}
\begin{split}
\frac{b_{1}}{2}\frac{d}{dt}\int_{\Gamma _{0}}\int_{-\infty }^{+\infty }&
|\phi (\xi ,t)|^{2}d\xi d\rho +b_{1}\int_{\Gamma _{0}}\int_{-\infty
}^{+\infty }(\xi ^{2}+\eta )|\phi (\xi ,t)|^{2}d\xi d\rho \\
& -b_{1}\int_{\Gamma _{0}}u_{t}(x,t)\int_{-\infty }^{+\infty }\mu (\xi )\phi
(\xi ,t)d\xi d\rho =0.
\end{split}
\label{3.4}
\end{equation}%
From $(3.1)$, $(3.3)$ and $(3.4)$ we obtain
\begin{equation*}
\frac{d}{dt}E(t)=-a\Arrowvert u_{t}\Arrowvert_{2}^{2}-\frac{1}{2}g\left(
t\right) \Arrowvert\nabla u\Arrowvert_{2}^{2}+\frac{1}{2}\left( g^{\prime
}\circ \nabla u\right) \left( t\right) -b_{1}\int_{\Gamma _{0}}\int_{-\infty
}^{+\infty }(\xi ^{2}+\eta )|\phi (\xi ,t)|^{2}d\xi d\rho \leq 0.
\end{equation*}
\end{proof}

\begin{lemma}
Assuming $(2.4)$ holds and that for all $(u_{0},u_{1},\phi _{0})\in
H_{\Gamma _{0}}^{1}(\Omega )\times L^{2}(\Omega )\times L^{2}(-\infty
,+\infty )$, verify
\begin{equation}
\left\{
\begin{array}{l}
\beta =C_{\ast }^{p}\left( \frac{2p}{p-2}E(0)\right) ^{\frac{p-2}{2}}<1 \\
I(u_{0})>0,%
\end{array}%
\right.  \label{3.5}
\end{equation}%
Then, $u(t)\in \aleph ,\quad \forall t\in \lbrack 0,T]$.
\end{lemma}

\begin{proof}
As $I(u_{0})>0$, there exists $T^{\ast }\leq T$ such that
\begin{equation*}
I(u)\geq 0, \ \ \forall t\in \lbrack 0,T^{\ast }).
\end{equation*}
This leads to:
\begin{equation}
\begin{split}
\left( 1-\int_{0}^{t}g(s)ds\right) \Arrowvert\nabla u\Arrowvert%
_{2}^{2}+\left( g\circ \nabla u\right) \left( t\right) & \leq \frac{2p}{p-2}%
J(t),\quad \forall t\in \lbrack 0,T^{\ast }) \\
& \leq \frac{2p}{p-2}E(0).
\end{split}
\label{3.6}
\end{equation}%
Using the Poincare inequality, $\left( 2.1\right)$, $\left( 2.3\right)$, $%
\left(3.5\right)$ and $(3.6)$, we obtain
\begin{equation}
\begin{split}
\Arrowvert u\Arrowvert_{p}^{p}& \leq C_{\ast }^{p}\Arrowvert\nabla u%
\Arrowvert_{2}^{p} \\
& \leq C_{\ast }^{p}\left( \frac{2p}{p-2}E(0)\right) ^{\frac{p-2}{2}}%
\Arrowvert\nabla u\Arrowvert_{2}^{2}.
\end{split}
\label{3.7}
\end{equation}%
Thus
\begin{equation*}
\left( 1-\int_{0}^{t}g(s)ds\right) \Arrowvert\nabla u\Arrowvert%
_{2}^{2}+\left( g\circ \nabla u\right) \left( t\right) -\Arrowvert u %
\Arrowvert_{p}^{p}>0, \ \forall t\in \lbrack 0,T^{\ast }).
\end{equation*}
Consequently $u\in H,\forall t\in \lbrack 0,T^{\ast })$.\newline
Repeating the procedure, $T^{\ast }$ can be extended to $T$, and that makes
the proof of our global existence result within reach.
\end{proof}

\begin{theorem}
Assume $(2.4)$ holds. Then for all
\begin{equation*}
(u_{0},u_{1},\phi _{0})\in H_{\Gamma _{0}}^{1}(\Omega )\times L^{2}(\Omega
)\times L^{2}(-\infty ,+\infty )
\end{equation*}
verifying $(3.5)$, the solution of system (\ref{2.21}) is global and bounded.
\end{theorem}

\begin{proof}
From $(3.2)$, we get
\begin{equation}
\begin{split}
E(0)\geq E(t)& =\frac{1}{2}\Arrowvert u_{t}\Arrowvert_{2}^{2}+\frac{1}{2}%
\left( 1-\int_{0}^{t}g(s)ds\right) \Arrowvert\nabla u\Arrowvert_{2}^{2}+%
\frac{1}{2}\left( g\circ \nabla u\right) \left( t\right) -\frac{1}{p}%
\Arrowvert u\Arrowvert_{p}^{p} \\
& +\frac{b_{1}}{2}\int_{\Gamma _{0}}\int_{-\infty }^{+\infty }|\phi (\xi
,t)|^{2}d\xi d\rho \\
& \geq \frac{1}{2}\Arrowvert u_{t}\Arrowvert_{2}^{2}+\frac{p-2}{2p}\Arrowvert%
\nabla u\Arrowvert_{2}^{2}+\frac{1}{p}I(t)+\frac{b_{1}}{2}\int_{\Gamma
_{0}}\int_{-\infty }^{+\infty }|\phi (\xi ,t)|^{2}d\xi d\rho .
\end{split}
\label{3.8}
\end{equation}%
Or $I(t)>0$, therefrom
\begin{equation*}
\Arrowvert u_{t}\Arrowvert_{2}^{2}+\Arrowvert\nabla u\Arrowvert%
_{2}^{2}+b_{1}\int_{\Gamma _{0}}\int_{-\infty }^{+\infty }|\phi (\xi
,t)|^{2}d\xi d\rho \leq C_{1}E(0),
\end{equation*}%
where $C_{1}=\max \{\frac{2}{b_{1}},\frac{2p}{p-2}, 2 \}$.
\end{proof}

\section{ Decay of solutions}

To proceed for the energy decay result, we construct an appropriate Lyapunov
functional as follows:

\begin{equation}
L(t)=\epsilon _{1}E(t)+\epsilon _{2}\psi _{1}(t)+\frac{\epsilon _{2}b_{1}}{2}%
\psi _{2}(t),  \label{4.1}
\end{equation}%
where
\begin{equation*}
\begin{split}
\psi _{1}(t)& =\int_{\Omega }u_{t}udx, \\
\psi _{2}(t)& =\int_{\Gamma _{0}}\int_{-\infty }^{+\infty }(\xi ^{2}+\eta
)\left( \int_{0}^{t}\phi (\xi ,s)ds\right) ^{2}d\xi d\rho ,
\end{split}%
\end{equation*}
and $\epsilon _{1}$, $\epsilon _{2}$ are positive constants.

\begin{lemma}
If $(u,\phi )$ is a regular solution of the problem (\ref{2.21}). Then, the
following equality holds
\begin{equation*}
\begin{split}
& \int_{\Gamma _{0}}\int_{-\infty }^{+\infty }(\xi ^{2}+\eta )\phi (\xi
,t)\int_{0}^{t}\phi (\xi ,s)dsd\xi d\rho = \\
& \int_{\Gamma _{0}}u(x,t)\int_{-\infty }^{+\infty }\phi (\xi ,t)\mu (\xi
)d\xi d\rho -\int_{\Gamma _{0}}\int_{-\infty }^{+\infty }|\phi (\xi
,t)|^{2}d\xi d\rho .
\end{split}%
\end{equation*}
\end{lemma}

\begin{proof}
From the second equation of (\ref{2.21}), we have
\begin{equation}
(\xi ^{2}+\eta )\phi (\xi ,t)=u_{t}(x,t)\mu (\xi )-\partial _{t}\phi (\xi
,t),\quad \forall x\in \Gamma _{0}.  \label{4.2}
\end{equation}%
Integrating (\ref{4.2}) over $\left[0, t\right]$, and using equations 3 and
6 from system (\ref{2.21}), we get
\begin{equation}
\int_{0}^{t}(\xi ^{2}+\eta )\phi (\xi ,s)ds=u(x,t)\mu (\xi )-\phi (\xi
,t),\quad \forall x\in \Gamma _{0},  \label{4.3}
\end{equation}%
hence,
\begin{equation}
(\xi ^{2}+\eta )\int_{0}^{t}\phi (\xi ,s)ds=u(x,t)\mu (\xi )-\phi (\xi
,t),\quad \forall x\in \Gamma _{0}.  \label{4.4}
\end{equation}%
A multiplying by $\phi $ followed by an integration over $\Gamma _{0}\times
(-\infty ,+\infty )$, leads to
\begin{equation*}
\begin{split}
& \int_{\Gamma _{0}}\int_{-\infty }^{+\infty }(\xi ^{2}+\eta )\phi (\xi
,t)\int_{0}^{t}\phi (\xi ,s)dsd\xi d\rho = \\
& \int_{\Gamma _{0}}u(x,t)\int_{-\infty }^{+\infty }\phi (\xi ,t)\mu (\xi
)d\xi d\rho -\int_{\Gamma _{0}}\int_{-\infty }^{+\infty }|\phi (\xi
,t)|^{2}d\xi d\rho .
\end{split}%
\end{equation*}
\end{proof}

\begin{lemma}
For any $(u,\phi )$ solution of problem (\ref{2.21}), we have
\begin{equation}
\alpha _{1}E(t)\leq L(t)\leq \alpha _{2}E(t),  \label{4.5}
\end{equation}
where $\alpha _{1}$,$\alpha _{2}$ are positive constants.
\end{lemma}

\begin{proof}
From $(4.3)$, we get
\begin{equation}
\int_{0}^{t}\phi (\xi ,s)ds=\frac{-\phi (\xi ,t)}{\xi ^{2}+\eta }+\frac{%
u(x,t)\mu (\xi )}{\xi ^{2}+\eta },\quad \forall x\in \Gamma _{0}.
\label{4.6}
\end{equation}%
Thus
\begin{equation}
\left( \int_{0}^{t}\phi (\xi ,s)ds\right) ^{2}=\frac{|\phi (\xi ,t)|^{2}}{%
(\xi ^{2}+\eta )^{2}}+\frac{|u(x,t)|^{2}\mu ^{2}(\xi )}{(\xi ^{2}+\eta )^{2}}%
-2\frac{\phi (\xi ,t)u(x,t)\mu (\xi )}{(\xi ^{2}+\eta )^{2}}.  \label{4.7}
\end{equation}%
A multiplying by $\xi ^{2}+\eta $ in $(4.7)$ followed by an integration over
$\Gamma _{0}\times (-\infty ,+\infty )$, leads to

\begin{equation}
\begin{split}
|\psi _{2}(t)|& \leq \int_{\Gamma _{0}}\int_{-\infty }^{+\infty }\frac{|\phi
(\xi ,t)|^{2}}{\xi ^{2}+\eta }d\xi d\rho +\int_{\Gamma
_{0}}|u(x,t)|^{2}\int_{-\infty }^{+\infty }\frac{\mu ^{2}(\xi )}{\xi
^{2}+\eta }d\xi d\rho \\
& +2\int_{\Gamma _{0}}\int_{-\infty }^{+\infty }\frac{|\phi (\xi
,t)u(x,t)\mu (\xi )|}{\xi ^{2}+\eta }d\xi d\rho .
\end{split}
\label{4.8}
\end{equation}

Using Young's inequality in order to have an estimation of the last term in $%
(4.8)$, we get for any $\delta >0$

\begin{equation}
\begin{split}
\int_{\Gamma _{0}}\int_{-\infty }^{+\infty }\frac{|\phi (\xi ,t)u(x,t)\mu
(\xi )|}{\xi ^{2}+\eta }d\xi d\rho & =\int_{\Gamma _{0}}\int_{-\infty
}^{+\infty }\frac{|\phi (\xi ,t)|}{(\xi ^{2}+\eta )^{\frac{1}{2}}}\frac{%
|u(x,t)\mu (\xi )|}{(\xi ^{2}+\eta )^{\frac{1}{2}}}d\xi d\rho \\
& \leq \frac{1}{4\delta }\int_{\Gamma _{0}}\int_{-\infty }^{+\infty }\frac{%
|\phi (\xi ,t)|^{2}}{\xi ^{2}+\eta }d\xi d\rho \\
& +\delta \int_{\Gamma _{0}}|u(x,t)|^{2}\int_{-\infty }^{+\infty }\frac{\mu
^{2}(\xi )}{\xi ^{2}+\eta }d\xi d\rho .
\end{split}
\label{4.9}
\end{equation}%
Combining $(4.9)$ and $(4.8)$, we obtain
\begin{equation}
\begin{split}
|\psi _{2}(t)|& \leq (\frac{2\delta +1}{2\delta })\int_{\Gamma
_{0}}\int_{-\infty }^{+\infty }\frac{|\phi (\xi ,t)|^{2}}{\xi ^{2}+\eta }%
d\xi d\rho \\
& +(2\delta +1)\int_{\Gamma _{0}}|u(x,t)|^{2}\int_{-\infty }^{+\infty }\frac{%
\mu ^{2}(\xi )}{\xi ^{2}+\eta }d\xi d\rho .
\end{split}
\label{4.10}
\end{equation}%
Since $\frac{1}{\xi ^{2}+\eta }\leq \frac{1}{\eta }$, then
\begin{equation}
\begin{split}
|\psi _{2}(t)& |\leq (\frac{2\delta +1}{2\delta \eta })\int_{\Gamma
_{0}}\int_{-\infty }^{+\infty }|\phi (\xi ,t)|^{2}d\xi d\rho \\
& +(2\delta +1)\int_{\Gamma _{0}}|u(x,t)|^{2}\int_{-\infty }^{+\infty }\frac{%
\mu ^{2}(\xi )}{\xi ^{2}+\eta }d\xi d\rho .
\end{split}
\label{4.11}
\end{equation}%
Applying Lammas 2 and 5 we get
\begin{equation}
|\psi _{2}(t)|\leq (\frac{2\delta +1}{2\delta \eta })\int_{\Gamma
_{0}}\int_{-\infty }^{+\infty }|\phi (\xi ,t)|^{2}d\xi d\rho
+A_{0}B_{q}(2\delta +1)\Arrowvert\nabla u\Arrowvert_{2}^{2}.  \label{4.12}
\end{equation}%
By Poincare-type inequality and Young's inequality, we obtain
\begin{equation}
|\psi _{1}(t)|\leq \frac{1}{2}\Arrowvert u_{t}\Arrowvert_{2}^{2}+\frac{%
C_{\ast }}{2}\Arrowvert\nabla u\Arrowvert_{2}^{2}.  \label{4.13}
\end{equation}%
Adding $(4.13)$ to $(4.12)$:
\begin{equation}
\begin{split}
|\psi _{1}(t)+\frac{b_{1}}{2}\psi _{2}(t)|& \leq |\psi _{1}(t)|+\frac{b_{1}}{%
2}|\psi _{2}(t)| \\
& \leq \frac{1}{2}\Arrowvert u_{t}\Arrowvert_{2}^{2}+\frac{1}{2}\left[
A_{0}B_{q}b_{1}(2\delta +1)+C_{\ast }\right] \Arrowvert\nabla u\Arrowvert%
_{2}^{2} \\
& +\frac{b_{1}}{2}\left[ \frac{2\delta +1}{2\delta \eta }\right]
\int_{\Gamma _{0}}\int_{-\infty }^{+\infty }|\phi (\xi ,t)|^{2}d\xi d\rho .
\end{split}
\label{4.14}
\end{equation}%
Therefore, By the energy definition given in $(3.1)$, for all $N>0$, we
have:
\begin{equation}
\begin{split}
|\psi _{1}(t)+\frac{b_{1}}{2}\psi _{2}(t)|& \leq NE(t)+\frac{1-N}{2}%
\Arrowvert u_{t}\Arrowvert_{2}^{2}+\frac{N}{p}\Arrowvert u_{t}\Arrowvert%
_{p}^{p} \\
& +\frac{1}{2}\left[ A_{0}B_{q}b_{1}(2\delta +1)+C_{\ast }-N\right] %
\Arrowvert\nabla u\Arrowvert_{2}^{2} \\
& +\frac{b_{1}}{2}\left[ \frac{2\delta +1}{2\delta \eta }-N\right]
\int_{\Gamma _{0}}\int_{-\infty }^{+\infty }|\phi (\xi ,t)|^{2}d\xi d\rho .
\end{split}
\label{4.15}
\end{equation}%
From $(3.7)$ and (\ref{4.15}), we finally get
\begin{equation}
\begin{split}
|\psi _{1}(t)+\frac{b_{1}}{2}\psi _{2}(t)|& \leq NE(t)+\frac{1-N}{2}%
\Arrowvert u_{t}\Arrowvert_{2}^{2} \\
& +\frac{1}{2}\left[ A_{0}B_{q}b_{1}(2\delta +1)+C_{\ast }-\frac{p-2}{2p}N%
\right] \Arrowvert\nabla u\Arrowvert_{2}^{2} \\
& +\frac{b_{1}}{2}\left[ \frac{2\delta +1}{2\delta \eta }-N\right]
\int_{\Gamma _{0}}\int_{-\infty }^{+\infty }|\phi (\xi ,t)|^{2}d\xi d\rho ,
\end{split}
\label{4.16}
\end{equation}%
where $N$ and $\epsilon _{1}$ are chosen as follows
\begin{equation*}
N>\max \{\frac{2\delta +1}{2\delta \eta },\ \frac{2p(A_{0}B_{q}b_{1}(2\delta
+1)+C_{\ast })}{p-2},\ 1\}
\end{equation*}%
\begin{equation*}
\epsilon _{1}\geq N\epsilon _{2}.
\end{equation*}%
Then, we conclude from (4.16)
\begin{equation*}
\alpha _{1}E(t)\leq L(t)\leq \alpha _{2}E(t),
\end{equation*}%
where
\begin{equation*}
\alpha _{1}=\epsilon _{1}-N\epsilon _{2}
\end{equation*}%
and
\begin{equation*}
\alpha _{2}=\epsilon _{1}+N\epsilon _{2}.
\end{equation*}
\end{proof}

Now, we prove the exponential decay of global solution.

\begin{theorem}
If $(2.4)$ and $(3.5)$ hold. Then, there exist $k$ and $K$, positive
constants such that the global solution of (\ref{2.21}) verifies
\begin{equation}
E(t)\leq Ke^{-kt}.  \label{4.17}
\end{equation}
\end{theorem}

\begin{proof}
By differentiation in $(4.1)$, we get
\begin{equation}
\begin{split}
L^{\prime }(t)& =\epsilon _{1}E^{\prime }(t)+\epsilon _{2}\Arrowvert u_{t}%
\Arrowvert_{2}^{2}+\epsilon _{2}\int_{\Omega }u_{tt}udx \\
& +\epsilon _{2}b_{1}\int_{\Gamma _{0}}\int_{-\infty }^{+\infty }(\xi
^{2}+\eta )\phi (\xi ,t)\int_{0}^{t}\phi (\xi ,s)dsd\xi d\rho .
\end{split}
\label{4.18}
\end{equation}%
Combining with (\ref{2.21}) to obtain
\begin{equation}
\begin{split}
L^{\prime }(t)& =\epsilon _{1}E^{\prime }(t)+\epsilon _{2}\left[ \Arrowvert %
u_{t}\Arrowvert_{2}^{2}-\Arrowvert\nabla u\Arrowvert_{2}^{2}+\Arrowvert u%
\Arrowvert_{p}^{p}-a\int_{\Omega }uu_{t}dx\right] \\
& -b_{1}\epsilon _{2}\int_{\Gamma _{0}}u(x,t)\int_{-\infty }^{+\infty }\mu
(\xi )\phi (\xi ,t)d\xi d\rho \\
& +b_{1}\epsilon _{2}\int_{\Gamma _{0}}\int_{-\infty }^{+\infty }(\xi
^{2}+\eta )\phi (\xi ,t)\int_{0}^{t}\phi (\xi ,s)dsd\xi d\rho .
\end{split}
\label{4.19}
\end{equation}%
An application of Lemma [8] leads to
\begin{equation}
\begin{split}
L^{\prime }(t)& =\epsilon _{1}E^{\prime }(t)+\epsilon _{2}\Arrowvert u_{t}%
\Arrowvert_{2}^{2}-\epsilon _{2}\Arrowvert\nabla u\Arrowvert%
_{2}^{2}+\epsilon _{2}\Arrowvert u\Arrowvert_{p}^{p} \\
& -b_{1}\epsilon _{2}\int_{\Gamma _{0}}\int_{-\infty }^{+\infty }|\phi (\xi
,t)|^{2}d\xi d\rho -a\epsilon _{2}\int_{\Omega }uu_{t}dx.
\end{split}
\label{4.20}
\end{equation}%
Using Poincare-type inequality and Young's inequality on the last term of $%
(4.20)$, we get for all $\delta ^{\prime }>0$
\begin{equation}
\int_{\Omega }uu_{t}dx\leq \frac{1}{4\delta ^{\prime }}\Arrowvert u_{t}%
\Arrowvert_{2}^{2}+C_{\ast }\delta ^{\prime }\Arrowvert\nabla u\Arrowvert%
_{2}^{2}.  \label{4.21}
\end{equation}%
From $(4.20)$, $(4.21)$ and $(3.2)$, we obtain
\begin{equation}
\begin{split}
L^{\prime }(t)& \leq \left[ -a\epsilon _{1}+\epsilon _{2}(1+\frac{a}{4\delta
^{\prime }})\right] \Arrowvert u_{t}\Arrowvert_{2}^{2}+\epsilon _{2}\left[
-1+\delta ^{\prime }C_{\ast }a\right] \Arrowvert\nabla u\Arrowvert_{2}^{2} \\
& +\epsilon _{2}\Arrowvert u\Arrowvert_{p}^{p}-b_{1}\epsilon
_{2}\int_{\Gamma _{0}}\int_{-\infty }^{+\infty }|\phi (\xi ,t)|^{2}d\xi
d\rho .
\end{split}
\label{4.22}
\end{equation}%
Using $(3.7)$ to get
\begin{equation}
\begin{split}
L^{\prime }(t)& \leq \left[ -a\epsilon _{1}+\epsilon _{2}(1+\frac{a}{4\delta
^{\prime }})\right] \Arrowvert u_{t}\Arrowvert_{2}^{2}+\epsilon _{2}\left[
-1+\delta ^{\prime }C_{\ast }a+C_{\ast }^{p}(\frac{2p}{p-2})^{\frac{p-2}{2}}%
\right] \Arrowvert\nabla u\Arrowvert_{2}^{2} \\
& -b_{1}\epsilon _{2}\int_{\Gamma _{0}}\int_{-\infty }^{+\infty }|\phi (\xi
,t)|^{2}d\xi d\rho .
\end{split}
\label{4.23}
\end{equation}%
On the other hand, from $(3.5)$
\begin{equation*}
-1+C_{\ast }^{p}(\frac{2p}{p-2})^{\frac{p-2}{2}}<0.
\end{equation*}%
For a small enough $\delta ^{\prime }$, we may have
\begin{equation*}
-1+\delta ^{\prime }C_{\ast }a+C_{\ast }^{p}(\frac{2p}{p-2})^{\frac{p-2}{2}%
}<0.
\end{equation*}%
Then, choosing $d>0$, depending only on $\delta ^{\prime }$ such that
\begin{equation}
\begin{split}
L^{\prime }(t& )\leq \left[ -a\epsilon _{1}+\epsilon _{2}(1+\frac{a}{4\delta
^{\prime }})\right] \Arrowvert u_{t}\Arrowvert_{2}^{2}-\epsilon _{2}d%
\Arrowvert\nabla u\Arrowvert_{2}^{2} \\
& -b_{1}\epsilon _{2}\int_{\Gamma _{0}}\int_{-\infty }^{+\infty }|\phi (\xi
,t)|^{2}d\xi d\rho .
\end{split}
\label{4.24}
\end{equation}%
Equivalently, for all positive constant $M$, we have
\begin{equation}
\begin{split}
L^{\prime }(t)& \leq \left[ -a\epsilon _{1}+\epsilon _{2}(1+\frac{a}{4\delta
^{\prime }}+\frac{M}{2})\right] \Arrowvert u_{t}\Arrowvert_{2}^{2}+\epsilon
_{2}\left[ \frac{M}{2}-d\right] \Arrowvert\nabla u\Arrowvert_{2}^{2} \\
& +b_{1}\epsilon _{2}\left[ \frac{M}{2}-1\right] \int_{\Gamma
_{0}}\int_{-\infty }^{+\infty }|\phi (\xi ,t)|^{2}d\xi d\rho -\epsilon
_{2}ME(t).
\end{split}
\label{4.25}
\end{equation}%
For $\epsilon _{1}$ and $M<\min \{2,2d\}$ chosen such that
\begin{equation*}
\epsilon _{1}>\frac{\epsilon _{2}(1+\frac{a}{4\delta ^{\prime }}+\frac{M}{2})%
}{a}.
\end{equation*}%
We obtain from $(4.25)$
\begin{equation}
L^{\prime }(t)\leq -M\epsilon _{2}E(t)\leq \frac{-\epsilon _{2}M}{\alpha _{2}%
}L(t),  \label{4.26}
\end{equation}%
as a result of $(4.5)$. Now, a simple integration of $(4.26)$ yields
\begin{equation*}
L(t)\leq L(0)e^{-kt},
\end{equation*}%
where $k=\frac{\epsilon _{2}M}{\alpha _{2}}$. Another use of $(4.5)$
provides $(4.17)$.
\end{proof}

\section{Blow up}

In the current section, we follow the same approach given in \cite{7} to
prove the blow up of solution of problem (\ref{2.21}).

\begin{remark}
By integration of (3.2) over $(0,t)$, we have
\begin{equation}
\begin{split}
E(t)& =E(0)-a\int_{0}^{t}\Arrowvert u_{s}\Arrowvert_{2}^{2}ds \\
& +\frac{1}{2}\left( 1-\int_{0}^{t}g(s)ds\right) \Arrowvert\nabla u\Arrowvert%
_{2}^{2}+\frac{1}{2}\left( g\circ \nabla u\right) \left( t\right) \\
& -b_{1}\int_{0}^{t}\int_{\Gamma _{0}}\int_{-\infty }^{+\infty }(\xi
^{2}+\eta )|\phi (\xi ,s)|^{2}d\xi d\rho ds.
\end{split}
\label{5.1}
\end{equation}
\end{remark}

Now, let us define $F(t)$:

\begin{eqnarray}
F(t) &=&\Arrowvert u\Arrowvert_{2}^{2}+a\int_{0}^{t}\Arrowvert u\Arrowvert%
_{2}^{2}ds  \notag \\
&&-\frac{1}{2}\left( 1-\int_{0}^{t}g(s)ds\right) \Arrowvert\nabla u\Arrowvert%
_{2}^{2}-\frac{1}{2}\left( g\circ \nabla u\right) \left( t\right) +b_{1}H(t),
\label{5.2}
\end{eqnarray}%
where
\begin{equation}
H(t)=\int_{0}^{t}\int_{\Gamma _{0}}\int_{-\infty }^{+\infty }(\xi ^{2}+\eta
)\left( \int_{0}^{s}\phi (\xi ,z)dz\right) ^{2}d\xi d\rho ds.  \label{5.3}
\end{equation}

\begin{lemma}
Assuming $\Arrowvert\nabla u\Arrowvert_{2}^{2}$ is bounded on $[0,T)$, Then
\begin{equation}
H(t)\leq C<+\infty .  \label{5.4}
\end{equation}%
More precisely
\begin{equation*}
H(t)\leq \frac{1}{2}C_{1}B_{q}e^{-\eta C_{2}}\left[ C_{2}^{2\alpha -1}\alpha
+C_{2}^{3-2\alpha }\eta \right] \Gamma (\alpha )T^{4}
\end{equation*}%
where
\begin{equation*}
C_{1}=\sup_{t\in \lbrack 0,T)}\{\Arrowvert\nabla u\Arrowvert_{2}^{2},1\}.
\end{equation*}
\end{lemma}

\begin{proof}
Using $(2.17)$ and $(2.18)$, we obtain
\begin{equation}
\phi (\xi ,t)=\int_{0}^{t}\mu (\xi )e^{-(\xi ^{2}+\eta )(t-s)}u(x,s)ds,\quad
\forall x\in \Gamma _{0}.  \label{5.5}
\end{equation}%
A H\"{o}lder inequality yields
\begin{equation}
\phi (\xi ,t)\leq \left( \int_{0}^{t}\mu ^{2}(\xi )e^{-2(\xi ^{2}+\eta
)(t-s)}ds\right) ^{\frac{1}{2}}\left( \int_{0}^{t}|u(x,s)|^{2}ds\right) ^{%
\frac{1}{2}},\quad \forall x\in \Gamma _{0}.  \label{5.6}
\end{equation}%
On the other hand,
\begin{equation}
\left( \int_{0}^{t}\phi (\xi ,s)ds\right) ^{2}\leq T\int_{0}^{t}|\phi (\xi
,s)|^{2}ds.  \label{5.7}
\end{equation}%
From $(5.6)$ in $(5.7)$, we obtain
\begin{equation}
\left( \int_{0}^{t}\phi (\xi ,s)ds\right) ^{2}\leq T\int_{0}^{t}\left[
\int_{0}^{s}\mu ^{2}(\xi )e^{-2(\xi ^{2}+\eta
)(s-z)}dz\int_{0}^{s}|u(x,z)|^{2}dz\right] ds.  \label{5.8}
\end{equation}%
Applying Lemma [2] leads to
\begin{equation}
\int_{\Gamma _{0}}\left( \int_{0}^{t}\phi (\xi ,s)ds\right) ^{2}d\rho \leq
B_{q}C_{1}T\int_{0}^{t}\left[ \int_{0}^{s}\mu ^{2}(\xi )e^{-2(\xi ^{2}+\eta
)(s-z)}dz\right] ds.  \label{5.9}
\end{equation}%
Since $z\in (0,s)$, we choose $\exists C_{2}\geq 0$ such that $s-z\geq \frac{%
C_{2}}{2}$ to term $(5.9)$ into
\begin{equation}
\int_{\Gamma _{0}}\left( \int_{0}^{t}\phi (\xi ,s)ds\right) ^{2}d\rho \leq
\frac{1}{2}B_{q}C_{1}T^{3}\mu ^{2}(\xi )e^{-C_{2}(\xi ^{2}+\eta )}.
\label{5.10}
\end{equation}%
A multiplication by $\xi ^{2}+\eta $ followed by integration over $%
(0,t)\times (-\infty ,+\infty )$, yields
\begin{equation}
\begin{split}
H(t)& \leq C_{1}B_{q}e^{-\eta C_{2}}T^{3}\int_{0}^{t}\left[
\int_{0}^{+\infty }\xi ^{2\alpha +1}e^{-C_{2}\xi ^{2}}d\xi \right] ds \\
& +C_{1}B_{q}e^{-\eta C_{2}}\eta T^{3}\int_{0}^{t}\left[ \int_{0}^{+\infty
}\xi ^{2\alpha -1}e^{-C_{2}\xi ^{2}}d\xi \right] ds.
\end{split}
\label{5.11}
\end{equation}%
Then
\begin{equation}
\begin{split}
H(t)& \leq \frac{1}{2}C_{1}B_{q}e^{-\eta C_{2}}C_{2}^{2\alpha
-1}T^{3}\int_{0}^{t}\left[ \int_{0}^{+\infty }y^{\alpha }e^{-y}dy\right] ds
\\
& +\frac{1}{2}C_{1}B_{q}e^{-\eta C_{2}}C_{2}^{3-2\alpha }\eta
T^{3}\int_{0}^{t}\left[ \int_{0}^{+\infty }y^{\alpha -1}e^{-y}dy\right] ds.
\end{split}
\label{5.12}
\end{equation}%
Applying a special integral ( Euler gamma function), we obtain
\begin{equation}
H(t)\leq \frac{1}{2}C_{1}B_{q}e^{-\eta C_{2}}\left[ C_{2}^{2\alpha -1}\alpha
+C_{2}^{3-2\alpha }\eta \right] \Gamma (\alpha )T^{4}.  \label{5.13}
\end{equation}
\end{proof}

\begin{lemma}
Suppose $p>2$, then
\begin{equation}
\begin{split}
& F^{\prime \prime }(t)\geq (p+2)\Arrowvert u_{t}\Arrowvert_{2}^{2} \\
& +2p\left\{ -E(0)+a\int_{0}^{t}\Arrowvert u_{s}\Arrowvert_{2}^{2}ds-\frac{1%
}{2}\left( 1-\int_{0}^{t}g(s)ds\right) \Arrowvert\nabla u\Arrowvert_{2}^{2}-%
\frac{1}{2}\left( g\circ \nabla u\right) \left( t\right) \right. . \\
& \left. +b_{1}\int_{0}^{t}\int_{\Gamma _{0}}\int_{-\infty }^{+\infty }(\xi
^{2}+\eta )|\phi (\xi ,s)|^{2}d\xi d\rho ds\right\}
\end{split}
\label{5.14}
\end{equation}
\end{lemma}

\begin{proof}
We differentiate with respect to $t$ in $(5.2)$, then we get
\begin{equation}
\begin{split}
F^{\prime }(t)& =2\int_{\Omega }uu_{t}dx+a\Arrowvert u\Arrowvert_{2}^{2} \\
& +\frac{1}{2}g\left( t\right) \Arrowvert\nabla u\Arrowvert_{2}^{2}-\frac{1}{%
2}\left( g^{\prime }\circ \nabla u\right) \left( t\right) \\
& +2b_{1}\int_{0}^{t}\int_{\Gamma _{0}}\int_{-\infty }^{+\infty }(\xi
^{2}+\eta )\phi (\xi ,s)\int_{0}^{s}\phi (\xi ,z)dzd\xi d\rho ds.
\end{split}
\label{5.15}
\end{equation}%
Using divergence theorem and (\ref{2.21}), we obtain
\begin{equation}
\begin{split}
F^{\prime \prime }(t)& =2\Arrowvert u_{t}\Arrowvert_{2}^{2}-2\int_{\Omega
}\nabla u\int_{0}^{t}g\left( t-s\right) \nabla u\left( s\right) dsdx \\
& +2\Arrowvert u\Arrowvert_{p}^{p}+2b_{1}\int_{\Gamma
_{0}}u(x,t)\int_{-\infty }^{+\infty }\mu (\xi )\phi (\xi ,t)d\xi d\rho \\
& +2b_{1}\int_{\Gamma _{0}}\int_{-\infty }^{+\infty }(\xi ^{2}+\eta )\phi
(\xi ,t)\int_{0}^{t}\phi (\xi ,s)dsd\xi d\rho .
\end{split}
\label{5.16}
\end{equation}%
By definition of of energy functional in $(3.1)$ and relation $(5.1)$, we
give the following evaluation of the third term of (\ref{5.16})
\begin{equation}
\begin{split}
2\Arrowvert u\Arrowvert_{p}^{p}& =p\Arrowvert u_{t}\Arrowvert_{2}^{2}+p%
\Arrowvert\nabla u\Arrowvert_{2}^{2}+pb_{1}\int_{\Gamma _{0}}\int_{-\infty
}^{+\infty }|\phi (\xi ,t)|^{2}d\xi d\rho -2pE(0) \\
& +2p\left[ a\int_{0}^{t}\Arrowvert u_{s}\Arrowvert_{2}^{2}ds-\frac{1}{2}%
\left( 1-\int_{0}^{t}g(s)ds\right) \Arrowvert\nabla u\Arrowvert_{2}^{2}-%
\frac{1}{2}\left( g\circ \nabla u\right) \left( t\right) \right. \\
& \left. +b_{1}\int_{0}^{t}\int_{\Gamma _{0}}\int_{-\infty }^{+\infty }(\xi
^{2}+\eta )|\phi (\xi ,s)|^{2}d\xi d\rho ds\right].
\end{split}
\label{5.17}
\end{equation}%
We can also estimate the last term of (\ref{5.16}) using Lemma [8]:
\begin{equation}
\begin{split}
& \int_{\Gamma _{0}}\int_{-\infty }^{+\infty }(\xi ^{2}+\eta )\phi (\xi
,t)\int_{0}^{t}\phi (\xi ,s)dsd\xi d\rho = \\
& \int_{\Gamma _{0}}u(x,t)\int_{-\infty }^{+\infty }\phi (\xi ,t)\mu (\xi
)d\xi d\rho -\int_{\Gamma _{0}}\int_{-\infty }^{+\infty }|\phi (\xi
,t)|^{2}d\xi d\rho .
\end{split}
\label{5.18}
\end{equation}%
From $(5.17)$, $(5.18)$ and $(5.16)$, we get
\begin{equation}
\begin{split}
& F^{\prime \prime }(t)\geq (p+2)\Arrowvert u_{t}\Arrowvert_{2}^{2}+(p-2)%
\Arrowvert\nabla u\Arrowvert_{2}^{2}+b_{1}(p-2)\int_{\Gamma
_{0}}\int_{-\infty }^{+\infty }|\phi (\xi ,t)|^{2}d\xi d\rho \\
& +2p\left[ -E(0)+a\int_{0}^{t}\Arrowvert u_{s}\Arrowvert_{2}^{2}ds-\frac{1}{%
2}(1-\int_{0}^{t}g(s)ds)\Arrowvert\nabla u\Arrowvert_{2}^{2}-\frac{1}{2}%
\left( g\circ \nabla u\right) \left( t\right) \right. . \\
& \left. +b_{1}\int_{0}^{t}\int_{\Gamma _{0}}\int_{-\infty }^{+\infty }(\xi
^{2}+\eta )|\phi (\xi ,s)|^{2}d\xi d\rho ds\right].
\end{split}
\label{5.19}
\end{equation}%
Taking $p>2$, we obtain the needed estimation
\begin{equation*}
\begin{split}
& F^{\prime \prime }(t)\geq (p+2)\Arrowvert u_{t}\Arrowvert_{2}^{2} \\
& +2p\left\{ -E(0)+a\int_{0}^{t}\Arrowvert u_{s}\Arrowvert_{2}^{2}ds-\frac{1%
}{2}\left( 1-\int_{0}^{t}g(s)ds\right) \Arrowvert\nabla u\Arrowvert_{2}^{2}-%
\frac{1}{2}\left( g\circ \nabla u\right) \left( t\right) \right. . \\
& \left. +b_{1}\int_{0}^{t}\int_{\Gamma _{0}}\int_{-\infty }^{+\infty }(\xi
^{2}+\eta )|\phi (\xi ,s)|^{2}d\xi d\rho ds\right\}
\end{split}%
\end{equation*}
\end{proof}

\begin{lemma}
Suppose $p>2$ and that either one of the next assumptions is verified
\newline
(i) $E(0)<0$.\newline
(ii) $E(0)=0$, and
\begin{equation}
F^{\prime }(0)>a\Arrowvert u_{0}\Arrowvert_{2}^{2}.  \label{5.20}
\end{equation}%
(iii) $E(0)>0$, and
\begin{equation}
F^{\prime }(0)>\left[ F(0)+l_{0}\right] +a\Arrowvert u_{0}\Arrowvert_{2}^{2},
\label{5.21}
\end{equation}%
where
\begin{equation*}
r=p-2\sqrt{p^{2}-p}
\end{equation*}%
and
\begin{equation}
l_{0}=a\Arrowvert u_{0}\Arrowvert_{2}^{2}-2E(0).  \label{5.22}
\end{equation}%
Then $F^{\prime }(t)>a\Arrowvert u_{0}\Arrowvert_{2}^{2}$, for $t>t_{0},$
where
\begin{equation}
t^{\ast }>\max \left\{ 0,\frac{F^{\prime }(0)-a\Arrowvert u_{0}\Arrowvert%
_{2}^{2}]}{2pE(0)}\right\} ,  \label{5.23}
\end{equation}%
where $t_{0}=t^{\ast }$ in case(i), and $t_{0}=0$ in case(ii) and (iii)
\end{lemma}

\begin{proof}
(i) Case of $E(0)<0$. \newline
From $(5.14)$, we have
\begin{equation*}
F^{^{\prime \prime }}(t)\geq -2pE(0),
\end{equation*}%
which clearly leads to :
\begin{equation*}
F^{^{\prime }}(t)\geq F^{^{\prime }}(0)-2pE(0)t.
\end{equation*}%
Then
\begin{equation*}
F^{^{\prime }}(t)>a\Arrowvert u_{0}\Arrowvert_{2}^{2},\ \forall t\geq
t^{\ast },
\end{equation*}%
where $t^{\ast }$ as given in $(5.23)$.\newline
(ii) Case $E(0)=0$. \newline
Using (5.14) we got
\begin{equation*}
F^{\prime \prime }(t)\geq 0,\quad \forall t\geq 0.
\end{equation*}%
Thus
\begin{equation*}
F^{\prime }(t)\geq F^{\prime }(0),\quad \forall t\geq 0.
\end{equation*}%
\newline
Then, by $(5.20)$
\begin{equation*}
F^{^{\prime }}(t)>a\Arrowvert u_{0}\Arrowvert_{2}^{2},\ \forall t\geq 0.
\end{equation*}%
(iii) Case $E(0)>0$.\newline
The proof of this case consist of getting to a differential inequality: $%
B^{\prime \prime }(t)-pB^{\prime }(t)+pB(t)\geq 0$ pursued by a use of Lemma
3. Indeed, from $(5.15)$ we have
\begin{equation}
\begin{split}
& F^{\prime }(t)=2\int_{\Omega }uu_{t}dx+a\Arrowvert u\Arrowvert_{2}^{2} \\
& +\frac{1}{2}g\left( t\right) \Arrowvert\nabla u\Arrowvert_{2}^{2}-\frac{1}{%
2}\left( g^{\prime }\circ \nabla u\right) \left( t\right) \\
& +2b_{1}\int_{0}^{t}\int_{\Gamma _{0}}\int_{-\infty }^{+\infty }(\xi
^{2}+\eta )\phi (\xi ,s)\int_{0}^{s}\phi (\xi ,z)dzd\xi d\rho ds.
\end{split}
\label{5.24}
\end{equation}%
\newline
Or, the last term in $(5.24)$ can be estimated using a Young's inequality
\begin{equation}
\begin{split}
& \int_{0}^{t}\int_{\Gamma _{0}}\int_{-\infty }^{+\infty }(\xi ^{2}+\eta
)\phi (\xi ,s)\int_{0}^{s}\phi (\xi ,z)dzd\xi d\rho ds \\
& \leq \frac{1}{2}\int_{0}^{t}\int_{\Gamma _{0}}\int_{-\infty }^{+\infty
}(\xi ^{2}+\eta )|\phi (\xi ,s)|^{2}d\xi d\rho ds \\
& +\frac{1}{2}\int_{0}^{t}\int_{\Gamma _{0}}\int_{-\infty }^{+\infty }(\xi
^{2}+\eta )\left( \int_{0}^{s}\phi (\xi ,z)dz\right) ^{2}d\xi d\rho ds
\end{split}
\label{5.25}
\end{equation}%
On the other hand
\begin{equation}
2\int_{0}^{t}\int_{\Omega }u_{s}udxds=\int_{0}^{t}\frac{d}{ds}\Arrowvert %
u_{s}\Arrowvert_{2}^{2}ds=\Arrowvert u\Arrowvert_{2}^{2}-\Arrowvert u_{0}%
\Arrowvert_{2}^{2}.  \label{5.26}
\end{equation}%
By Young's inequality, we get
\begin{equation}
\Arrowvert u\Arrowvert_{2}^{2}\leq \int_{0}^{t}\Arrowvert u_{s}\Arrowvert%
_{2}^{2}ds+\int_{0}^{t}\Arrowvert u\Arrowvert_{2}^{2}ds+\Arrowvert u_{0}%
\Arrowvert_{2}^{2}.  \label{5.27}
\end{equation}%
Now, we remake $(5.24)$ using $(5.25)$ and $(5.27)$

\begin{equation}
\begin{split}
F^{\prime }(t)& \leq \Arrowvert u\Arrowvert_{2}^{2}+\Arrowvert u_{t}%
\Arrowvert_{2}^{2}+a\int_{0}^{t}\Arrowvert u_{s}\Arrowvert%
_{2}^{2}ds+a\int_{0}^{t}\Arrowvert u\Arrowvert_{2}^{2}ds+a\Arrowvert u_{0}%
\Arrowvert_{2}^{2} \\
& -\frac{1}{2}\left( 1-\int_{0}^{t}g(s)ds\right) \Arrowvert\nabla u\Arrowvert%
_{2}^{2}-\frac{1}{2}\left( g\circ \nabla u\right) \left( t\right)
+b_{1}\int_{0}^{t}\int_{\Gamma _{0}}\int_{-\infty }^{+\infty }(\xi ^{2}+\eta
)|\phi (\xi ,s)|^{2}d\xi d\rho ds \\
& +b_{1}\int_{0}^{t}\int_{\Gamma _{0}}\int_{-\infty }^{+\infty }(\xi
^{2}+\eta )\left( \int_{0}^{s}\phi (\xi ,z)dz\right) ^{2}d\xi d\rho ds.
\end{split}
\label{5.28}
\end{equation}%
From definition of $F$ in $(5.2)$, inequality $(5.28)$ also becomes
\begin{equation}
\begin{split}
F^{\prime }(t)\leq & F(t)+\Arrowvert u_{t}\Arrowvert_{2}^{2}+b_{1}%
\int_{0}^{t}\int_{\Gamma _{0}}\int_{-\infty }^{+\infty }(\xi ^{2}+\eta
)|\phi (\xi ,s)|^{2}d\xi d\rho ds \\
& +a\int_{0}^{t}\Arrowvert u_{s}\Arrowvert_{2}^{2}ds+a\Arrowvert u_{0}%
\Arrowvert_{2}^{2}.
\end{split}
\label{5.29}
\end{equation}%
Thus by $(5.14)$, we get
\begin{equation}
\begin{split}
F^{\prime \prime }(t)-p\left\{ F^{\prime }(t)-F(t)\right\} & \geq 2%
\Arrowvert u_{t}\Arrowvert_{2}^{2}+ap\int_{0}^{t}\Arrowvert u_{s}\Arrowvert%
_{2}^{2}ds-pa\Arrowvert u_{0}\Arrowvert_{2}^{2}-2pE(0) \\
& +pb_{1}\int_{0}^{t}\int_{\Gamma _{0}}\int_{-\infty }^{+\infty }(\xi
^{2}+\eta )|\phi (\xi ,s)|^{2}d\xi d\rho ds.
\end{split}
\label{5.30}
\end{equation}%
Hence
\begin{equation}
F^{\prime \prime }(t)-pF^{\prime }(t)+pF(t)+pl_{0}\geq 0,  \label{5.31}
\end{equation}%
where
\begin{equation*}
l_{0}=a\Arrowvert u_{0}\Arrowvert_{2}^{2}-2E(0).
\end{equation*}%
Posing
\begin{equation*}
B(t)=F(t)+l_{0}.
\end{equation*}%
Leads to
\begin{equation}
B^{\prime \prime }(t)-pB^{\prime }(t)+pB(t)\geq 0.  \label{5.32}
\end{equation}%
By Lemma (3) and for $p=\delta +1$, we conclude that if
\begin{equation}
B^{\prime }(t)>(p-2\sqrt{p^{2}-p})B(0)+a\Arrowvert u_{0}\Arrowvert_{2}^{2}.
\label{5.33}
\end{equation}%
Then
\begin{equation*}
F^{\prime }(t)=B^{\prime }(t)>a\Arrowvert u_{0}\Arrowvert_{2}^{2}\quad
\forall t\geq 0.
\end{equation*}
\end{proof}

\begin{theorem}
Suppose $p>2$ and that either one of the next assumptions is verified
\newline
(i) $E(0)<0$. \newline
(ii) $E(0)=0$ and (5.20) holds.\newline
(iii) $0<E(0)<\frac{(2p-4)\left( F^{\prime }(t_{0})-a\Arrowvert u_{0}%
\Arrowvert_{2}^{2}\right) ^{2}J(t_{0})^{\frac{1}{\gamma _{1}}}}{16p}$ and
(5.21) holds. \newline
Then, In the sense of Definition 1, the solution $(u,\phi )$ blows up at
finite time $T^{\ast }$.\newline
For case (i):
\begin{equation}
T^{\ast }\leq t_{0}-\frac{J(t_{0})}{J^{\prime }(t_{0})}.  \label{5.34}
\end{equation}%
Moreover, if $J(t_{0})<\min \left\{ 1,\sqrt{\frac{\sigma }{-b}}\right\} ,$
we get
\begin{equation}
T^{\ast }\leq t_{0}+\frac{1}{\sqrt{-b}}\ln \frac{\sqrt{\frac{\sigma }{-b}}}{%
\sqrt{\frac{\sigma }{-b}}-J(t_{0})}.  \label{5.35}
\end{equation}%
For case (ii): we get either
\begin{equation}
T^{\ast }\leq t_{0}-\frac{J(t_{0})}{J^{\prime }(t_{0})},  \label{5.36}
\end{equation}%
or
\begin{equation}
T^{\ast }\leq t_{0}+\frac{J(t_{0})}{J^{\prime }(t_{0})}.  \label{5.37}
\end{equation}%
For case (iii):
\begin{equation}
T^{\ast }\leq \frac{J(t_{0})}{\sqrt{\sigma }},  \label{5.38}
\end{equation}%
or else
\begin{equation}
T^{\ast }\leq t_{0}+2^{\frac{3\gamma _{1}+1}{2\gamma _{1}}}\frac{\gamma _{1}c%
}{\sqrt{\sigma }}\{1-[1-cJ(t_{0})]^{\frac{1}{2\gamma _{1}}}\},  \label{5.39}
\end{equation}%
where $\gamma_{1}=\frac{p-4}{4}$, $c=(\frac{b}{\sigma })^{\frac{\gamma _{1}}{%
2+\gamma _{1}}}$, $J(t)$, $b$ and $\sigma $ are as in $(5.40)$ and $(5.54)$
respectively.\newline
\end{theorem}

Note that $t_0 =0$ in cases (ii) and (iii). For case (i), we have as in
(5.23): $t_0=t^*$.

\begin{proof}
Consider
\begin{equation}
J(t)=\left[ F(t)+a(T-t)\Arrowvert u_{0}\Arrowvert_{2}^{2}\right] ^{-\gamma
_{1}},\quad t\in \lbrack t_{0},T].  \label{5.40}
\end{equation}%
We differentiate on $J(t)$ to get
\begin{equation}
J^{^{\prime }}(t)=-\gamma _{1}J(t)^{1+\frac{1}{\gamma _{1}}}\left[ F^{\prime
}(t)-a\Arrowvert u_{0}\Arrowvert_{2}^{2}\right]  \label{5.41}
\end{equation}%
and again
\begin{equation}
J^{^{\prime \prime }}(t)=-\gamma _{1}J(t)^{1+\frac{2}{\gamma _{1}}}G(t),
\label{5.42}
\end{equation}%
where
\begin{equation}
G(t)=F^{^{\prime \prime }}(t)\left[ F(t)+a(T-t)\Arrowvert u_{0}\Arrowvert%
_{2}^{2}\right] -(1+\gamma _{1})\left\{ F^{^{\prime }}(t)-a\Arrowvert u_{0}%
\Arrowvert_{2}^{2}\right\} ^{2}.  \label{5.43}
\end{equation}%
Using $(5.14)$, we obtain
\begin{equation*}
\begin{split}
& F^{\prime \prime }(t)\geq (p+2)\Arrowvert u_{t}\Arrowvert_{2}^{2} \\
& +2p\left\{ -E(0)+a\int_{0}^{t}\Arrowvert u_{s}\Arrowvert_{2}^{2}ds-\frac{1%
}{2}\left( 1-\int_{0}^{t}g(s)ds\right) \Arrowvert\nabla u\Arrowvert_{2}^{2}-%
\frac{1}{2}\left( g\circ \nabla u\right) \left( t\right) \right. . \\
& \left. +b_{1}\int_{0}^{t}\int_{\Gamma _{0}}\int_{-\infty }^{+\infty }(\xi
^{2}+\eta )|\phi (\xi ,s)|^{2}d\xi d\rho ds\right\}
\end{split}%
\end{equation*}%
Consequently
\begin{equation}
\begin{split}
& F^{\prime \prime }(t)\geq -2pE(0) \\
& p\left\{ \Arrowvert u_{t}\Arrowvert_{2}^{2}+a\int_{0}^{t}\Arrowvert u_{s}%
\Arrowvert_{2}^{2}ds-\frac{1}{2}\left( 1-\int_{0}^{t}g(s)ds\right) \Arrowvert%
\nabla u\Arrowvert_{2}^{2}-\frac{1}{2}\left( g\circ \nabla u\right) \left(
t\right) \right. \\
& \left. +b_{1}\int_{0}^{t}\int_{\Gamma _{0}}\int_{-\infty }^{+\infty }(\xi
^{2}+\eta )|\phi (\xi ,s)|^{2}d\xi d\rho ds\right\}.
\end{split}
\label{5.44}
\end{equation}%
Or, from from $(5.15)$ and the fact that $\Arrowvert u\Arrowvert_{2}^{2}-%
\Arrowvert u_{0}\Arrowvert_{2}^{2}=2\int_{0}^{t}\int_{\Omega }u_{s}udxds$,
we attain
\begin{equation}
\begin{split}
F^{\prime }(t)-a\Arrowvert u_{0}\Arrowvert_{2}^{2}& =2\int_{\Omega
}uu_{t}dx+2a\int_{0}^{t}\int_{\Omega }u_{s}udxds \\
& +2b_{1}\int_{0}^{t}\int_{\Gamma _{0}}\int_{-\infty }^{+\infty }(\xi
^{2}+\eta )\phi (\xi ,s)\int_{0}^{s}\phi (\xi ,z)dzd\xi d\rho ds.
\end{split}
\label{5.45}
\end{equation}%
Going back to $(5.43)$ with $(5.44)$ and $(5.45)$ in hand, we get

\begin{equation}
\begin{split}
& G(t)\geq -2pE(0)J(t)^{\frac{-1}{\gamma _{1}}} \\
& +p\left\{ \Arrowvert u_{t}\Arrowvert_{2}^{2}+a\int_{0}^{t}\Arrowvert u_{s}%
\Arrowvert_{2}^{2}ds-\frac{1}{2}\left( 1-\int_{0}^{t}g(s)ds\right) \Arrowvert%
\nabla u\Arrowvert_{2}^{2}-\frac{1}{2}\left( g\circ \nabla u\right) \left(
t\right) \right. \\
& \left. +b_{1}\int_{0}^{t}\int_{\Gamma _{0}}\int_{-\infty }^{+\infty }(\xi
^{2}+\eta )|\phi (\xi ,s)|^{2}d\xi d\rho ds\right\} \\
& \times \left[ \Arrowvert u\Arrowvert_{2}^{2}+a\int_{0}^{t}\Arrowvert u%
\Arrowvert_{2}^{2}ds-\frac{1}{2}\left( 1-\int_{0}^{t}g(s)ds\right) \Arrowvert%
\nabla u\Arrowvert_{2}^{2}-\frac{1}{2}\left( g\circ \nabla u\right) \left(
t\right) \right. \\
& \left. +b_{1}\int_{0}^{t}\int_{\Gamma _{0}}\int_{-\infty }^{+\infty }(\xi
^{2}+\eta )\left( \int_{0}^{s}\phi (\xi ,z)dz\right) ^{2}d\xi d\rho ds\right]
\\
& -4(1+\gamma _{1})\left\{ \int_{\Omega }uu_{t}dx+a\int_{0}^{t}\int_{\Omega
}u_{s}udxds+\frac{1}{2}g\left( t\right) \Arrowvert\nabla u\Arrowvert_{2}^{2}-%
\frac{1}{2}\left( g^{\prime }\circ \nabla u\right) \left( t\right) \right. .
\\
& \left. +b_{1}\int_{0}^{t}\int_{\Gamma _{0}}\int_{-\infty }^{+\infty }(\xi
^{2}+\eta )\phi (\xi ,s)\int_{0}^{s}\phi (\xi ,z)dzd\xi d\rho ds\right\}
^{2}.
\end{split}
\label{5.46}
\end{equation}%
For sake of simplicity, we introduce the following notations

\begin{equation*}
\begin{split}
\mathbf{A}& =\Arrowvert u\Arrowvert_{2}^{2}+a\int_{0}^{t}\Arrowvert u%
\Arrowvert_{2}^{2}ds-\frac{1}{2}\left( 1-\int_{0}^{t}g(s)ds\right) \Arrowvert%
\nabla u\Arrowvert_{2}^{2}-\frac{1}{2}\left( g\circ \nabla u\right) \left(
t\right) \\
& +b_{1}\int_{0}^{t}\int_{\Gamma _{0}}\int_{-\infty }^{+\infty }(\xi
^{2}+\eta )\left( \int_{0}^{s}\phi (\xi ,z)dz\right) ^{2}d\xi d\rho ds, \\
\mathbf{B}& =\int_{\Omega }uu_{t}dx+a\int_{0}^{t}\int_{\Omega }u_{s}udxds+%
\frac{1}{2}g\left( t\right) \Arrowvert\nabla u\Arrowvert_{2}^{2}-\frac{1}{2}%
\left( g^{\prime }\circ \nabla u\right) \left( t\right) \\
& +b_{1}\int_{0}^{t}\int_{\Gamma _{0}}\int_{-\infty }^{+\infty }(\xi
^{2}+\eta )\phi (\xi ,s)\int_{0}^{s}\phi (\xi ,z)dzd\xi d\rho ds, \\
\mathbf{C}& =\Arrowvert u_{t}\Arrowvert_{2}^{2}+a\int_{0}^{t}\Arrowvert u_{s}%
\Arrowvert_{2}^{2}ds-\frac{1}{2}\left( 1-\int_{0}^{t}g(s)ds\right) \Arrowvert%
\nabla u\Arrowvert_{2}^{2}-\frac{1}{2}\left( g\circ \nabla u\right) \left(
t\right) \\
& +b_{1}\int_{0}^{t}\int_{\Gamma _{0}}\int_{-\infty }^{+\infty }(\xi
^{2}+\eta )|\phi (\xi ,s)|^{2}d\xi d\rho ds.
\end{split}%
\end{equation*}%
Therefore
\begin{equation}
Q(t)\geq -2pE(0)J(t)^{\frac{-1}{\gamma _{1}}}+p\left\{ \mathbf{A}\mathbf{C}-%
\mathbf{B}^{2}\right\} .  \label{5.47}
\end{equation}%
Note that, $\forall w\in R$ and $\forall t>0$,

\begin{equation}
\begin{split}
\mathbf{A}w^{2}+2\mathbf{B}w+\mathbf{C}& =\left[ w^{2}\Arrowvert u\Arrowvert%
_{2}^{2}+2w\int_{\Omega }uu_{t}dx+\Arrowvert u_{t}\Arrowvert_{2}^{2}\right]
\\
& +a\int_{0}^{t}\left[ w^{2}\Arrowvert u\Arrowvert_{2}^{2}+2w\int_{\Omega
}uu_{s}dx+\Arrowvert u_{s}\Arrowvert_{2}^{2}\right] ds \\
& +\left( w^{2}+1\right) \left( -\frac{1}{2}\left(
1-\int_{0}^{t}g(s)ds\right) \Arrowvert\nabla u\Arrowvert_{2}^{2}-\frac{1}{2}%
\left( g\circ \nabla u\right) \left( t\right) \right) \\
& +w\left( \frac{1}{2}g\left( t\right) \Arrowvert\nabla u\Arrowvert_{2}^{2}-%
\frac{1}{2}\left( g^{\prime }\circ \nabla u\right) \left( t\right) \right) \\
& +b_{1}\int_{0}^{t}\int_{\Gamma _{0}}\int_{-\infty }^{+\infty }(\xi
^{2}+\eta )\left[ w^{2}\left( \int_{0}^{s}\phi (\xi ,z)dz\right) ^{2}\right.
\\
& \quad \quad \quad \left. +2w\phi (\xi ,s)\int_{0}^{s}\phi (\xi ,z)dz+|\phi
(\xi ,s)|^{2}\right] d\xi d\rho ds.
\end{split}
\label{5.48}
\end{equation}%
Hence

\begin{equation}
\begin{split}
\mathbf{A}w^{2}& +2\mathbf{B}w+\mathbf{C}=\left[ w\Arrowvert u\Arrowvert_{2}+%
\Arrowvert u_{t}\Arrowvert_{2}\right] ^{2}+a\int_{0}^{t}\left[ w\Arrowvert u%
\Arrowvert_{2}+\Arrowvert u_{s}\Arrowvert_{2}\right] ^{2}ds \\
& +\left( w^{2}+1\right) \left( -\frac{1}{2}\left(
1-\int_{0}^{t}g(s)ds\right) \Arrowvert\nabla u\Arrowvert_{2}^{2}-\frac{1}{2}%
\left( g\circ \nabla u\right) \left( t\right) \right) +w\left( \frac{1}{2}%
g\left( t\right) \Arrowvert\nabla u\Arrowvert_{2}^{2}-\frac{1}{2}\left(
g^{\prime }\circ \nabla u\right) \left( t\right) \right) \\
& +b_{1}\int_{0}^{t}\int_{\Gamma _{0}}\int_{-\infty }^{+\infty }(\xi
^{2}+\eta )\left[ w\int_{0}^{s}\phi (\xi ,z)dz+|\phi (\xi ,s)|\right]
^{2}d\xi d\rho ds.
\end{split}
\label{5.49}
\end{equation}%
It is clear that
\begin{equation*}
\mathbf{A}w^{2}+2\mathbf{B}+\mathbf{C}\geq 0
\end{equation*}%
and
\begin{equation}
\mathbf{B}^{2}-\mathbf{A}\mathbf{C}\leq 0.  \label{5.50}
\end{equation}%
Then, from $(5.47)$ and $(5.50)$, we obtain
\begin{equation}
G(t)\geq -2pE(0)J(t)^{\frac{-1}{\gamma _{1}}},\quad t\geq t_{0}.
\label{5.51}
\end{equation}%
Hence, by $(5.42)$ and $(5.51)$
\begin{equation}
J^{\prime \prime }(t)\leq \frac{p^{2}-4p}{2}E(0)J(t)^{1+\frac{1}{\gamma _{1}}%
},\quad t\geq t_{0}.  \label{5.52}
\end{equation}%
Or, by Lemma [12], $J^{^{\prime }}(t)<0$, where $t\geq t_{0}$.

A multiplication by $J^{^{\prime }}(t)$ in $(5.52)$, followed by an
integration from $t_{0}$ to $t$ leads to

\begin{equation}
J^{^{\prime }}(t)^{2}\geq \sigma +bJ(t)^{2+\frac{1}{\gamma _{1}}},
\label{5.53}
\end{equation}%
where
\begin{equation}
\left\{
\begin{split}
\sigma & =\left[ \frac{(p-4)^{2}}{16}\left( F^{^{\prime }}(t_{0})-\Arrowvert %
u_{0}\Arrowvert_{2}^{2}\right) ^{2}-\frac{p(p-4)^{2}}{2p-4}E(0)J(t_{0})^{%
\frac{-1}{\gamma _{1}}}\right] J(t_{0})^{2+\frac{2}{\gamma _{1}}} \\
b& =\frac{p(p-4)^{2}}{2p-4}E(0).
\end{split}%
\right.  \label{5.54}
\end{equation}%
Note that $\sigma >0$, is equivalent to $E(0)<\frac{(2p-4)\left( F^{\prime
}(t_{0})-a\Arrowvert u_{0}\Arrowvert_{2}^{2}\right) ^{2}J(t_{0})^{\frac{1}{%
\gamma _{1}}}}{16p}$, which by Lemma [4] ensure the existence of a finite
time $T^{\ast }>0$ such that
\begin{equation*}
\lim_{t\rightarrow T^{\ast -}}J\left( t\right) =0.
\end{equation*}%
That involves

\begin{equation}
\lim_{t\rightarrow T^{\ast -}}\left[ \Arrowvert u\Arrowvert%
_{2}^{2}+a\int_{0}^{t}\Arrowvert u\Arrowvert_{2}^{2}ds-\frac{1}{2}\left(
1-\int_{0}^{t}g(s)ds\right) \Arrowvert\nabla u\Arrowvert_{2}^{2}-\frac{1}{2}%
\left( g\circ \nabla u\right) \left( t\right) +b_{1}H(t)\right] ^{-1}=0.
\label{5.55}
\end{equation}%
i.e.
\begin{equation}
\lim_{t\rightarrow T^{\ast -}}\left[ \Arrowvert u\Arrowvert%
_{2}^{2}+a\int_{0}^{t}\Arrowvert u\Arrowvert_{2}^{2}ds-\frac{1}{2}\left(
1-\int_{0}^{t}g(s)ds\right) \Arrowvert\nabla u\Arrowvert_{2}^{2}-\frac{1}{2}%
\left( g\circ \nabla u\right) \left( t\right) +b_{1}H(t)\right] =+\infty .
\label{5.56}
\end{equation}%
So, there exists a $T$ such that $t_{0}<T\leq T^{*}$ and $\Arrowvert\nabla u%
\Arrowvert_{2}^{2}\longrightarrow +\infty $ as $t\longrightarrow T^{-}$.%
\newline

Indeed, if it is not the case, then $\Arrowvert\nabla u\Arrowvert_{2}^{2}$
remained bounded on $[t_{0},T^{*})$, which by Lemma [10] leads to
\begin{equation*}
\lim_{t\rightarrow T^{* -}}\left[ \Arrowvert u\Arrowvert_{2}^{2}+b_{1}H(t)%
\right] =C<+\infty ,
\end{equation*}%
contradicting (\ref{5.56}).
\end{proof}

\section{Conclusion}

Much attention has been accorded to fractional partial differential
equations during the past two decades due to the many chemical engineering,
biological, ecological and electromagnetism phenomena that are modeled by
initial boundary value problems with fractional boundary conditions. In the
context of boundary dissipations of fractional order problems, the main
research focus is on asymptotic stability of solutions starting by writing
the equations as an augmented system. Then, various techniques are used such
as LaSalle's invariance principle and multiplier method mixed with frequency
domain. we prove in this paper under suitable conditions on the initial data
the stability of a wave equation with fractional damping and memory term.
This technique of proof was recently used by \cite{R20} to study the
exponential decay of a system of nonlocal singular viscoelastic equations.
\newline
Here we also considered three different cases on the sign of the initial
energy as recently examined by Zarai and al \cite{20}, where they studied
the blow up of a system of nonlocal singular viscoelastic equations.\newline
In next work, we try to extend the same study of this paper to a general
source term case. \newline
\newline
\textbf{\textit{\ Acknowledgement}} \newline
For any decision, the authors are grateful to the anonymous referees for the
careful reading and their important observations/suggestions for sake of
improving this paper. In memory of Mr. Mahmoud ben Mouha Boulaaras
(1910--1999).


\begin{thebibliography}{99}
\bibitem{1} Aassila M, Cavalcanti MM, Domingos Cavalcanti VN. Existence and
uniform decay of the wave equation with nonlinear boundary damping and
boundary memory source term. Calc Var Partial Differ Equ. 2002;15:155-180.
https://link.springer.com/article/10.1007/s005260100096

\bibitem{2} Akil M, Wehbe A. Stabilization of multidimensional wave equation
with locally boundary fractional dissipation law under geometric condition.
Math Control Relat Fields. 2019; 9(1):97-116.
http://dx.doi.org/10.3934/mcrf.2019005.

\bibitem{3} Achouri Z, Amroun NE, Benaissa A. The Euler-Bernoulli beam
equation with boundary dissipation of fractional derivative type. Math
Methods Appl Sci. 2017;40:3837-3854. https://doi.org/10.1002/mma.4267.

\bibitem{4} Alabau-Boussouira F, Pr\"{u}ss J, Zacher R. Exponential and
polynomial stability of a wave equation for boundary memory damping with
singular kernels. C R Acad Sci Paris Ser. 2009;I347:277-282.
https://doi.org/10.1016/j.crma.2009.01.005.

\bibitem{5} Alabau-Boussouira F. asymptotic stability of wave equations with
memory and fractional boundary dampings. Applications Mathematicae.
2008;35:247-258. https://dx.doi.org/10.4064/am35-3-01.

\bibitem{6} Blanc E, Chiavassa G, Lombard B. Biot-JKD model: Simulation of
1D transient poroelastic waves with fractional derivatives. J Comput Phys.
2013;237:1-20. https://doi.org/10.1016/j.jcp.2012.12.003.

\bibitem{7} Gala, S., Ragusa, M. A: Logarithmically improved regularity
criterion for the Boussinesq equations in Besov spaces with negative
indices, Applicable Analysis \textbf{95} (6), 1271-1279, (2016)

\bibitem{8} Dai H, Zhang H. Exponential growth for wave equation with
fractional boundary dissipation and boundary source term. Boundary Value
Prob. 2014;2014(1):1-8. https://doi.org/10.1186/s13661-014-0138-y.

\bibitem{9} Draifia A, Zarai A, Boulaaras S. Global Existence and Decay of
Solutions of a Singular Nonlocal Viscoelastic System. Rend Circ Mat Palermo
II Ser. 2018. https://doi.org/10.1007/s12215-018-00391-z.

\bibitem{10} Gala, S., Liu, Q., Ragusa, M. A: A new regularity criterion for
the nematic liquid crystal fows, Applicable Analysis \textbf{91} (9),
1741-1747 (2012).

\bibitem{11} Kirane M, Tatar Ne. Non-existence results for a semilinear
hyperbolic problem with boundary conditions of memory type. Journal of
Analysis and Its Applications. 2000;19:453-468.

\bibitem{12} Alizadeh, M., Alimohammady, M: Regularity and entropy solutions
of some elliptic equations, Miskolc Mathematical Notes \textbf{19}(2),
(2018), 715-729.

\bibitem{13} Polidoro, S., Ragusa, M. A: Harnack inequality for hypoelliptic
ultraparabolic equa- tions with a singular lower order term, Revista
Matematica Iberoamericana 24 (3), 1011-1046 (2008).

\bibitem{14} Mbodje B. Wave energy decay under fractional derivative
controls. IMA Journal of Mathematical Control and Information.
2006;23:237-257. https://doi.org/10.1093/imamci/dni056

\bibitem{15} Mu\~{n}oz-V\'{a}zquez AJ, Parra-Vega V, S\'{a}nchez-Orta A,
Romero-Galv\'{a}n R. Output feedback finite-time stabilization of systems
subject to H\"{o}lder disturbances via continuous fractional sliding modes.
Math Prbl Eng. 2017.https://doi.org/10.1155/2017/3146231

\bibitem{16} Matsuyama T, Ikerata R. On global solutions and energy decay
for the wave equations of Kirchhoff type with nonlinear damping terms. J
Math Anal Appl. 1996;204:729-753. https://doi.org/10.1006/jmaa.1996.0464

\bibitem{17} Naifar O, Makhlouf A, Hammami M, Chen L. Global practical
Mittag Leffler stabilization by output feedback for a class of nonlinear
fractional-order systems. Asian J Control. 2018;20(1):599-607.
https://doi.org/10.1002/asjc.1576.

\bibitem{R20} Aounallah R, Boulaaras S, Zarai A and Cherif B, General decay
and blow up of solution for a nonlinear wave equation with a fractional
boundary damping, Math Methods Appl Sci. 2020;
https://doi.org/10.1002/mma.6455

\bibitem{18} Tarasov VE. Fractional Dynamics: Applications of Fractional
Calculus to Dynamics of Particles. New York: Fields and Media, Springer;
2011. https://doi.org/10.1007/978-3-642-14003-7.

\bibitem{19} Val\'{e}rio D, Machado J, Kiryakova V. Some pioneers of the
applications of fractional calculus. Frac Calc Appl Anal.
2014;17(2):552-578. https://doi.org/10.2478/s13540-014-0185-1.

\bibitem{20} Zarai A, Draifia A, Boulaaras S. Blow up of solutions for a
system of nonlocal singular viscoelastic equations. Appl Anal.
2018;97:2231-2245. https://doi.org/10.1080/00036811.2017.1359564.

\bibitem{21} Zhou HC, Guo BZ. Boundary feedback stabilization for an
unstable time fractional reaction diffusion equation. SIAM J. Control Optim.
2018;56: 75-101. https://doi.org/10.1137/15M1048999.
\end{thebibliography}
\end{document}